\documentclass[oneside,leqno,11pt]{article}
\usepackage{amssymb,amsmath,latexsym,amsthm}

\newtheorem{theorem}[equation]{Theorem}

\newtheorem{lemma}[equation]{Lemma}

\newtheorem{corollary}[equation]{Corollary}

\newtheorem{proposition}[equation]{Proposition}

\newtheorem{example}[equation]{Example}

\newtheorem{remark}[equation]{Remark}

\def\C{\mathbb{C}}
\def\H{\mathbb{H}}
\def\R{\mathbb{R}}
\def\Ad{{\rm Ad}\,}
\def\diag{{\rm diag}\,}
\def\ad{{\rm ad}\,}
\def\sideremark#1{\ifvmode\leavevmode\fi\vadjust{\vbox to0pt{\vss
 \hbox to 0pt{\hskip\hsize\hskip1em
\vbox{\hsize2cm\tiny\raggedright\pretolerance10000 
 \noindent #1\hfill}\hss}\vbox to8pt{\vfil}\vss}}} 

\title{Toward a Classification of Killing Vector Fields of Constant Length 
	on Pseudo--Riemannian Normal Homogeneous Spaces}

\author{Joseph A. Wolf\footnote{Address: Department of Mathematics, 
	University of California, Berkeley,
        CA 94720--3840, USA; e--mail: {\tt jawolf@math.berkeley.edu}.
        Research partially supported by a Simons Foundation grant and by the
        Dickson Emeriti Professorship at the University of California,
	Berkeley.}, \,\, 
Fabio Podest\` a\footnote{Address: Dipartimento di Matematica 
	e Informatica "U.Dini", Viale Morgagni 67/A, I-Firenze, Italy;
	e--mail: {\tt podesta@math.unifi.it}} \,\,\&
Ming Xu\footnote{Address: College of Mathematics,
        Tianjin Normal University,
        Tianjin 300387, P.R.China; e--mail: {\tt mgmgmgxu@163.com}.
        Research supported by NSFC no. 11271216, State Scholarship
        Fund of CSC (no. 201408120020), Science and Technology Development
	Fund for Universities and Colleges in Tianjin
	(no. 20141005), Doctor fund of Tianjin Normal
        University (no. 52XB1305). Corresponding author.}}

\date{March 27, 2015}

\begin{document}

\maketitle

\begin{abstract}
In this paper we develop the basic tools for a classification of Killing
vector fields of constant length on pseudo--riemannian homogeneous spaces.
This extends a recent paper of M. Xu and J. A. Wolf, 
which classified the pairs $(M,\xi)$ where
$M = G/H$ is a Riemannian normal homogeneous space, $G$ is a compact simple
Lie group, and $\xi \in \mathfrak{g}$ defines a 
nonzero Killing vector field of 
constant length on $M$.  The method there was direct computation.  Here we 
make use of the moment map $M \to \mathfrak{g}^*$ and the flag manifold
structure of $\Ad(G)\xi$ to give a shorter, more geometric proof which 
does not require compactness and which is valid in the pseudo--riemannian 
setting.  In that context we break the classification problem into three parts.
The first is easily settled.  The second concerns the cases where $\xi$ is
elliptic and $G$ is simple (but not necessarily compact); that case is our
main result here.  The third, which remains open, is a more combinatorial
problem involving elements of the first two.
\end{abstract}

\section{Introduction}\label{sec1}
\setcounter{equation}{0}
We consider a connected real reductive Lie group $G$, a nondegenerate 
invariant bilinear form $b$ on $\mathfrak{g}$, and a closed reductive 
subgroup $H$ in $G$ such 
that $b$ is nondegenerate on $\mathfrak{h}$.  Decompose
$\mathfrak{g} = \mathfrak{h} + \mathfrak{m}$ where $\mathfrak{m}$ is the
$b$--orthocomplement of $\mathfrak{h}$.  Then $b$ is nondegenerate on
$\mathfrak{m}$ and induces a pseudo--riemannian metric $ds^2$ on $M = G/H$.
Those are our {\em normal} pseudo--riemannian metrics.  This includes
the Riemannian case, where $ds^2$ is either positive definite (as usual) or 
negative definite (so that $b$ can be the Killing form when $G$ is a
compact semisimple Lie group).  Note the dependence
on the pair $(G,b)$.  If $G'$ is another transitive group of isometries of
$(M,ds^2)$ then $ds^2$ need not be normal as a homogeneous space of $G'$.
\smallskip

Let $\xi \in \mathfrak{g}$.  It induces a Killing vector field on $M$ 
which we denote $\xi^M$.  If $x \in M$ then $\xi^M_x$ is the corresponding
tangent vector at $x$.  We say that $\xi^M$ has {\em constant length} 
(perhaps pseudo--length would be a better term) if the function 
$x \mapsto ds^2(\xi^M_x,\xi^M_x)$ is constant on $M$.  The goal of this
paper is the classification of triples $(G,H,\xi)$ where 
$\xi \in \mathfrak{g}$ is nonzero and elliptic, and where $\xi^M$ has
constant length.
\smallskip

In the setting
of pseudo--riemannian manifolds, constant length Killing vector fields 
(also called Clifford--Killing or CK vector fields; see \cite{XW2014}) 
are the appropriate replacement for isometries of 
constant displacement (CW isometries). 
\smallskip

In Section \ref{sec2} we discuss a flag manifold $G_\C/Q$
that connects the moment map for conjugation orbits in $\mathfrak{g}$ 
with the length function for $\xi^M$.  Then in Section \ref{sec3} 
we develop a method of passage through the complex domain that carries 
this connection to flag domains and the pseudo--riemannian setting.
In Section \ref{sec4} we use these tools to carry out the classification
for the cases where $G_\C$ is simple; the main result is Theorem 
\ref{abs-simple-summary}.  Those tools don't apply directly
to the case where $G$ is simple but $G_\C$ is not, but in Section \ref{sec5}
we use other methods to carry out the classification; there the
main result is Theorem \ref{not-abs-simple}.  Section \ref{sec6}
summarizes these classifications to give one of the two main results
of this paper, Theorem \ref{simple-summary}.  As a consequence of these
classifications, Corollary \ref{homog} indicates the pseudo--riemannian
analog of the correspondence between homogeneity for quotient manifolds
and isometries of constant displacement.
\smallskip

The other principal result
is Theorem \ref{reduce-to-simple}, which in effect describes current
progress toward a classification where $G$ need not be simple.

Let $pr_\mathfrak{h}$ and $pr_\mathfrak{m}$ denote the respective orthogonal 
projections of $\mathfrak{g}$ to $\mathfrak{h}$ and $\mathfrak{m}$.  Then
$ds^2(\xi^M_x,\xi^M_x) = 
b(pr_\mathfrak{m}(\Ad(g)\xi),pr_\mathfrak{m}(\Ad(g)\xi))$
where $x = gH$.  Since $b(\Ad(g)\xi,\Ad(g)\xi)$ is independent of $g\in G$,
and 
$$
b(\Ad(g)\xi,\Ad(g)\xi) = 
b(pr_\mathfrak{h}(\Ad(g)\xi),pr_\mathfrak{h}(\Ad(g)\xi)) 
+ b(pr_\mathfrak{m}(\Ad(g)\xi),pr_\mathfrak{m}(\Ad(g)\xi)),
$$ 
\begin{lemma} \label{m-equiv}
Let $\xi \in \mathfrak{g}$.  Then $\xi^M$ has constant length if and 
only if
$$
f_\xi(g) := b(pr_\mathfrak{h}(\Ad(g)\xi),pr_\mathfrak{h}(\Ad(g)\xi))
$$
is independent of $g \in G$.
\end{lemma}

In view of Lemma \ref{m-equiv} and our assumption that $G$ is connected, the
constant length property for $\xi^M$ depends only on the pair $(\mathfrak{g},
\mathfrak{h})$.  Thus we can (and will) be casual about passing to and from
covering groups of $G$ and about connectivity of $H$.  In practise this
will be only a matter of whether it is more convenient to write $Spin$ or $SO$.

\section{The Flag Domain}\label{sec2}
\setcounter{equation}{0}
We use $b$ to identify adjoint orbits of $G$ on $\mathfrak{g}$
and coadjoint orbits of $G$ on $\mathfrak{g}^*$.

\begin{proposition}\label{parab}
Suppose that $\xi \in \mathfrak{g}$ is elliptic, in other words that
$\ad(\xi)$ is semisimple {\rm (diagonalizable over $\C$)} with
pure imaginary eigenvalues.  Let $L$ denote the centralizer of $\xi$ in $G$.
Then $G_\C$ has a parabolic subgroup $Q$ with the properties
\begin{itemize}
\item $L$ is the isotropy subgroup of $G$ at the base point $z_0 = 1Q$
for the action of $G$ {\rm (as a subgroup of $G_\C$)}
on the complex flag manifold $Z = G_\C/Q$,
\item $L_\C$ is the reductive part of $Q$,
\item the orbit $G(z_0) \subset Z$ is open and carries a $G$--invariant
pseudo--K\" ahler metric, which can be normalized so that
\item $\Ad(g)\xi \mapsto gQ$ is a symplectomorphism of
$\mathcal{O}_\xi := \Ad(G)\xi$ onto $G(z_0)$ where the symplectic form
on $\mathcal{O}_\xi$ is the Kostant--Souriau form 
$\omega(\eta,\zeta) = b(\xi, [\eta,\zeta])$ and the symplectic form on
$G(z_0)$ is the imaginary part of the invariant pseudo--K\" ahler metric.
\end{itemize}
In particular $\mathcal{O}_\xi$ has a $G$--invariant pseudo--K\" ahler
structure.
\end{proposition} 

\begin{proof}  By construction $L$ is reductive.  In fact $\xi$ is contained
in a fundamental (maximally compact) Cartan subalgebra $\mathfrak{t}$ of
$\mathfrak{g}$, and $\mathfrak{l}_\C$ is $\mathfrak{t}_\C$ plus all
the $\mathfrak{t}_\C$--root spaces $\mathfrak{g}_\alpha$ for roots $\alpha$ 
that vanish on $\xi$.  Define $\mathfrak{q} = \mathfrak{l}_\C +
\sum_{\alpha(i\xi) < 0}\,\mathfrak{g}_\alpha$\,.  It is a parabolic
subalgebra of $\mathfrak{g}_\C$ with reductive part $\mathfrak{l}_\C$\,.
\smallskip

Let $\tau$ denote complex conjugation of $\mathfrak{g}_\C$ over
$\mathfrak{g}$.  Then $\tau(i\xi) = -i\xi$ so $\mathfrak{q} +
\tau\mathfrak{q} = \mathfrak{g}_\C$\,, and also
$\mathfrak{q} \cap \tau\mathfrak{q} = \mathfrak{l}_\C$\,.
There are two immediate
consequences: (i)  $G(z_0)$ is open in $Z = G_\C/Q$ and (ii)
$\Ad(g)\xi \mapsto gQ$ is a diffeomorphism of $\mathcal{O}_\xi$
onto $G(z_0)$.  Note that (ii) uses simple connectivity of both
$Z$ and $\mathcal{O}_\xi$\,.
\smallskip

Since $\mathfrak{q} \cap \tau\mathfrak{q} = \mathfrak{l}_\C$\,,
which is reductive, $G(z_0)$ carries a
$G$--invariant measure.  Any such measure comes from the volume form of
an invariant indefinite--K\" ahler metric; see \cite{W1969}, or see the 
exposition of flag domains in \cite{FHW},  This metric is constructed
in \cite{W1969} using an invariant bilinear form; as is the
Kostant--Souriau form, and by the construction a proper normalization of 
the metric has imaginary part equal to the Kostant--Souriau form.
\end{proof}

\begin{remark}\label{biliotti}
In our flag domain cases, {\rm Proposition \ref{parab}} extends the structural
result of {\rm \cite[Theorem 1.3(4)]{B2007}} from symplectic to 
pseudo--K\" ahler.
\hfill $\diamondsuit$
\end{remark}

\begin{remark} \label{conjugates}
The parabolic $\mathfrak{q}$ is the sum of the non-positive eigenspaces of
$\ad(i\xi)$ on $\mathfrak{g}_\C$\,.  If $g \in G_\C$ now $\Ad(g)\mathfrak{q}$
is the sum of the non-positive eigenspaces of $\ad(\Ad(g)\xi)$ on
$\mathfrak{g}_\C$\,.  As $Q$ is its own normalizer in $G_\C$ we can identify
$Z = G_\C/Q$ with the space of $\Ad(G_\C)$--conjugates of $\mathfrak{q}$.
Thus, if $S$ is any subgroup of $G_\C$\,, we see exactly 
how $\Ad(S)\xi \subset Z$.
\hfill $\diamondsuit$
\end{remark}

We are using $b$ to identify $\mathfrak{g}$ with $\mathfrak{g}^*$; similarly
use $b|_{\mathfrak{h}}$ to identify $\mathfrak{h}$ with $\mathfrak{h}^*$.
The inclusion
\begin{equation}\label{mu-G}
\mu_G: \mathcal{O}_\xi \hookrightarrow \mathfrak{g}
\end{equation}
coincides with the moment map for the (necessarily Hamiltonian) action
of $G$ on $\mathcal{O}_\xi$.  Now consider the action of $H$ on 
$\mathcal{O}_\xi$\,.  The corresponding moment map is
\begin{equation}\label{mu-H}
\mu_H := pr_{\mathfrak{h}}\circ\mu_G:\mathcal{O}_\xi\to\mathfrak{h}.
\end{equation}
Thus Lemma \ref{m-equiv} can be reformulated as
\begin{lemma}\label{mom-const}
Let $\xi \in \mathfrak{g}$.  Then $\xi^M$ has constant length 
if and only if
$$
\zeta \mapsto b(\mu_H(\zeta),\mu_H(\zeta))
\text{ is constant on } \mathcal{O}_\xi\,.
$$ 
\end{lemma}

\section{Holomorphic Considerations}\label{sec3}
\setcounter{equation}{0}
The group $H$ is reductive in $G$ because $b$ is nondegenerate on 
$\mathfrak{h}$.  Thus \cite{M1955} there is a Cartan involution $\theta$ of $G$
such that $\theta|_H$ is a Cartan involution on $H$.  That gives us
the decompositions
$$
\mathfrak{g} = \mathfrak{k} + \mathfrak{p} \text{ and }
  \mathfrak{h} = (\mathfrak{h} \cap \mathfrak{k}) +
        (\mathfrak{h} \cap \mathfrak{p})
$$
into $\pm 1$ eigenspaces of $d\theta$.
\smallskip

From now on we suppose that $G$ is semisimple and that $b$ is a
positive linear combination of the Killing forms of the
simple ideals of $\mathfrak{g}$. Thus $b$ is negative definite on 
$\mathfrak{k}$ and positive definite on $\mathfrak{p}$.  
\{The reader can extend many of our results 
to the case of reductive $G$ by stipulating $b(\mathfrak{k},\mathfrak{p}) = 0$,
$b$ negative definite on $\mathfrak{k}$, and $b$ positive definite on
$\mathfrak{p}$.\}  
The decompositions of $\mathfrak{g}$ and $\mathfrak{h}$ give us
compact real forms
$$
\mathfrak{g}_u = \mathfrak{k} + i\mathfrak{p} 
	= \mathfrak{h}_u + \mathfrak{m}_u \text{ where }
  \mathfrak{h}_u = (\mathfrak{h} \cap \mathfrak{k}) + 
        i(\mathfrak{h} \cap \mathfrak{p}) \text{ and }
  \mathfrak{m}_u = (\mathfrak{m} \cap \mathfrak{k}) + 
	i(\mathfrak{m} \cap \mathfrak{p})
$$
of $\mathfrak{g}_\C$\, $\mathfrak{h}_\C$ and $\mathfrak{m}_\C$\,.  
Let $G_u$ and $H_u$ denote
the compact real forms of $G_\C$ and $H_\C$ corresponding to
$\mathfrak{g}_u$ and $\mathfrak{h}_u$\,.  
\smallskip

Extend $b$ to a $\C$--bilinear form $b_\C$ on $\mathfrak{g}_\C$. 
Then $b_u := b_\C|_{\mathfrak{g}_u}$ is negative definite.  As
$b(\mathfrak{h},\mathfrak{m}) = 0$ we have 
$b_\C(\mathfrak{h}_\C,\mathfrak{m}_\C) = 0$ and thus
$b_u(\mathfrak{h}_u,\mathfrak{m}_u) = 0$.  The orthogonal projection
$pr_{\mathfrak{h}_\C}: \mathfrak{g}_\C \to \mathfrak{h}_\C$ restricts to
orthogonal projection $pr_{\mathfrak{h}_u}: \mathfrak{g}_u \to \mathfrak{h}_u$.
\smallskip

\begin{lemma}\label{holo} Define $f_\xi : G_\C \to \C$ by 
$f_\xi(g) = b(pr_{\mathfrak{h}_\C}(\Ad(g)\xi),pr_{\mathfrak{h}_\C}(\Ad(g)\xi))$. 
Then $f_\xi$ is holomorphic.
\end{lemma}

\begin{proof} The map $g \mapsto \Ad(g)\xi$ is holomorphic on $G_\C$\,,
the projection $pr_{\mathfrak{h}_\C}: \mathfrak{g}_\C \to \mathfrak{h}_\C$
is holomorphic, and $b$ is complex bilinear.
\end{proof}

\begin{lemma}\label{const}
If $\xi^M$ has constant length then $f_\xi : G_\C \to \C$ is constant, and
in particular $f_\xi|_{G_u}$ is constant.
\end{lemma}

\begin{proof}
If $\xi^M$ has constant length then $f_\xi$ is constant on $G$.  Since
$G$ is a real form of $G_\C$ and $f_\xi$ is holomorphic, it follows that
$f_\xi$ is constant.
\end{proof}

Denote $M_u = G_u/H_u$ where $M_u$ carries the normal homogeneous
Riemannian metric defined by $b_u|_{\mathfrak{m}_u}$\,.  In effect it is
the natural compact real form of the affine algebraic variety $M_\C = G_\C/H_\C$
dual to $M = G/H$.  If $\xi \in \mathfrak{k}$, in particular if
$\xi \in \mathfrak{g}_u$, we write $\xi^{M_u}$ for the corresponding
vector field on $M_u$\,. Now Lemmas \ref{mom-const} and \ref{const} give us

\begin{proposition}\label{carryover}
If $\xi \in \mathfrak{k}$ and $\xi^M$ has constant length on $M$ if, and
only if, $\xi^{M_u}$ has constant length on the Riemannian normal 
homogeneous space $M_u := G_u/H_u$\,.
\end{proposition}

\section{Classification for $G_\C$ Simple}\label{sec4}
\setcounter{equation}{0}
In this section we carry out the classification of constant
length Killing vector fields $\xi^M$, on reductive normal homogeneous 
pseudo--riemannian manifolds $M = G/H$ when the group $G_\C$ is simple.
The compact version of this classification was done by direct computation
in \cite{XW2014}, but here we have a less computational approach that
starts with classification (\cite{O1962}, or see \cite{O1994}) of
Onischik for irreducible complex flag manifolds $Z = G_u/L_u$\,, on which a
proper closed subgroup $H_u$ of $G_u$ acts transitively.  
On the other hand we need the classification where $H$ need not be compact.
For that we use methods from \cite{W2001}.  In Section \ref{sec5} we 
give a separate
argument to deal with the case where $G$ is simple but $G_\C$ is not.
Then in Section \ref{sec6}
we translate those results to the classification of constant
length Killing vector fields $\xi^M$ on reductive normal homogeneous
pseudo--riemannian manifolds $M = G/H$, with $G$ simple and $\xi$ nonzero
and elliptic.
\smallskip

For clarity of exposition we always assume that $G_\C$ is connected and 
simply connected, that the real forms $G$ and $G_u$ are analytic subgroups
of $G_\C$, and that $H$, $H_\C$ and $H_u$ are analytic subgroups of $G$,
$G_\C$ and $G_u$\,.

\begin{proposition}\label{irred-flags-compact} {\rm \cite{O1962}}
Consider a complex flag manifold $Z = G_\C/Q$. Suppose that $Z$ is
irreducible, i.e., that $G_\C$ is simple. Then the closed connected subgroups 
$H_u \subset G_u$ transitive on $Z$, $\{1\} \ne H_u  \subsetneqq G_u$, are 
precisely those given as follows.
\smallskip

{\rm 1.} $Z = SU(2n)/U(2n-1) = P^{2n-1}(\C)$, complex projective 
$(2n-1)$--space; there $G_\C = SL(2n;\C)$ and $H_u = Sp(n)$.
\smallskip

{\rm 2.} $Z = SO(2n + 2)/U(n + 1)$, unitary structures on $\R^{2n+2}$; 
there $G_\C = SO(2n + 2;\C)$ and $H_u = SO(2n + 1)$. 
\smallskip

{\rm 3.} $Z = Spin(7)/(Spin(5) \cdot Spin(2))$, nonsingular complex 
quadric; there $G_\C=Spin(7;\C)$ and $H_u$ is the compact exceptional 
group $G_2$\,. 
\smallskip
\end{proposition}

Here is the noncompact version of Proposition \ref{irred-flags-compact}.

\begin{proposition}\label{irred-flags-nonc}
Consider a complex flag manifold $Z = G_\C/Q$ with $G_\C$ simple.
Here is a complete list of the connected subgroups $H \subset G$ with 
$H \ne \{1\}$ and $H_u$ transitive on $Z$.
\smallskip

\noindent {\rm 1.} $Z = SU(2n)/U(2n-1) = P^{2n-1}(\C)$ and $H_u = Sp(n)$.
Then $(G,H)$ is one of

{\rm (i)} $(SU(2p,2q),Sp(p,q))$ with $p + q = n$ or 

{\rm (ii)} $(SL(2n;\R),Sp(n;\R))$.
\smallskip

\noindent {\rm 2.} $Z = SO(2n + 2)/U(n + 1)$ and $H_u = SO(2n + 1)$.
Then $(G,H)$ must be 

{\rm (i)} $(SO(2p+1,2q+1),SO(2p+1,2q))$ with $p+q = n$ or 

{\rm (ii)} $(SO(2p+2,2q),SO(2p+1,2q))$ with $p+q = n$.

\smallskip

\noindent {\rm 3.} $Z = Spin(7)/(Spin(5) \cdot Spin(2))$ and $H_u = G_2$\,.
Then the pair $(G,H)$ must be 

{\rm (i)} $(Spin(7),G_2)$ or 

{\rm (ii)} $(Spin(3,4),(G_2)_\R)$.  $($Here $(G_2)_\R$ is the split real 
form of $(G_2)_\C)$.
\end{proposition}

\begin{proof} Suppose $Z = SU(2n)/U(2n-1;\C) = P^{2n-1}(\C)$ and $H_u = Sp(n)$.
The real forms of $(H_u)_\C = Sp(n;\C)$ are the $Sp(p,q)$, $p+q = n$, 
and $Sp(n;\R)$, and the real forms of $(G_u)_\C = SL(2n;\C)$ are the $SU(r,s)$, 
$r + s = 2n$ and the special linear groups $SL(2n;\R)$ and $SL(n;\H)$.  
\smallskip 

If $G = SU(r,s)$ and $J = \left ( \begin{smallmatrix}
I_r & 0 \\ 0 & -I_s \end{smallmatrix} \right )$ then $G = \{g \in SL(2n;\C)
\mid g\cdot J \cdot {^t \bar g} = J\}$.  Thus $H \ne Sp(n;\R)$, for that group
cannot have both a symmetric and an antisymmetric bilinear invariant on
$\R^{2n}$.  Now $G = SU(r,s)$ implies $H = Sp(p,q)$, which in turn implies
$G = SU(2p,2q)$.
Also, if $G = SL(2n;\R)$ then $H \not\cong Sp(p,q)$ so 
$H = Sp(n;\R)$.  
\smallskip

Next suppose $Z = SO(2n + 2)/U(n + 1)$ and $H_u = SO(2n + 1)$.  
The real forms of $(H_u)_\C = SO(2n+1;\C)$ are the $SO(r,s)$ with $r+s = 2n+1$,
and the real forms of $(G_u)_\C = SO(2n+2;\C)$ are the $SO(k,\ell)$ with
$k+\ell = 2n+2$ and $SO^*(2n+2)$.  The maximal compact subgroup of $SO^*(2n+2)$
is $U(n+1)$, which does not contain any $SO(r)\times SO(s)$ with $r+s = 2n+1$;
so $G \ne SO^*(2n+2)$.   Thus $G = SO(k,\ell)$ and $H = SO(r,s)$ with
$r \leqq k, s \leqq \ell$ and $k + \ell = r+s+1$, as asserted.
\smallskip

Finally suppose $Z = Spin(7)/(Spin(5) \cdot Spin(2))$ and $H_u = G_2$\,.
The real forms of $(G_u)_\C = Spin(7;\C)$ are the $Spin(a,b)$ with $a + b = 7$,
and the real forms of $(H_u)_\C = (G_2)_\C$ are the compact form $G_2$
and the split form $(G_2)_\R$\,.  Now \cite[Theorem 3.1]{W1964} completes
the argument that $G/H$ is $Spin(7)/G_2$ or $Spin(3,4)/(G_2)_\R$\,.
\end{proof}

Now we summarize, include the case where $H_u$ acts trivially on $Z$,
and note that one case is eliminated by the requirement that
$\xi \in \mathfrak{g}$.
\begin{theorem}\label{abs-simple-summary}
Suppose that $G$ is absolutely simple, i.e. that $G_\C$ is simple.
Then there is a
nonzero elliptic element $\xi \in \mathfrak{g}$ such that the Killing
vector field $\xi^M$ on the normal homogeneous space
$M = G/H$ has constant length, if and only if, up to finite covering,
$(G,H)$ is one of the following pairs.
\smallskip

\noindent {\rm 1.} $Z = SU(2n)/U(2n-1) = P^{2n-1}(\C)$ and $H_u = Sp(n)$.
Then $(G,H)$ is one of the

$(SU(2p,2q),Sp(p,q))$ with $p + q = n$, or is 
$(SL(2n;\R),Sp(n;\R))$.
\smallskip

\noindent {\rm 2.} $Z = SO(2n)/U(n)$ and $H_u = SO(2n - 1)$.
Then $(G,H)$ is one of the

$(SO(2p,2q),SO(2p-1,2q))$ with $p+q = n$.
\smallskip

\noindent {\rm 3.} $Z = Spin(7)/(Spin(5) \cdot Spin(2))$ and $H_u = G_2$\,.
Then $(G,H)$ is 

$(Spin(7),G_2)$ or $(Spin(3,4),(G_2)_\R)$.
\smallskip

\noindent {\rm 4.} $\mathfrak{h}  = 0$ and $(G,H)$ is the group manifold
pair $(G,\{1\})$.
\end{theorem}

\begin{proof}
Retain the notation of Section \ref{sec3}.  We can suppose $\xi \in
\mathfrak{k} \subset \mathfrak{g}_u$\,.  By Proposition \ref{carryover},\,
$\xi$ induces a Killing vector field $\xi^{M_u}$ of constant length on
the normal homogeneous Riemannian manifold $M_u = G_u/H_u$\,.
The adjoint orbit $Z := \Ad(G_u)\xi \subset \mathfrak{g}_u$ is
endowed with the $G_u$--invariant symplectic structure given by the
Kostant--Souriau form.  The $b$--orthogonal projection
$pr_{\mathfrak{h}}:Z \to \mathfrak{h}_u$ defines a moment map $\mu$
for the Hamiltonian action of $H_u$ on $Z$.  By hypothesis $\mu$ has 
constant length with respect to $b|_{\mathfrak{h}_u}$\,, and by
\cite{GP2004} the flag manifold $Z$ is a K\"ahler product
$Z_1 \times Z_2$ with $H_u$ acting transitively on $Z_1$ and trivially
on $Z_2$\,.  Since $G_\C$ is simple, either $Z = Z_1$ or $Z = Z_2$.\,,
and if $H_u$ is not trivial then $H_u$ acts transitively on $Z$.  We
have shown that $(G,H)$ either is a group manifold or is one of the pairs 
listed in Propositions \ref{irred-flags-compact} and \ref{irred-flags-nonc}.
\smallskip

In the cases listed in Proposition \ref{irred-flags-compact}, i.e. the
cases where $G$ is compact, we already have nonzero elliptic elements
$\xi \in \mathfrak{g}$ such that the centralizer of $\xi$ in $G$ is
transitive on $G/H$.  For $G/H = SU(2n)/Sp(n)$ we use 
$\xi_1 = \sqrt{-1}\,\diag\{-(2n-1),I_{2n-1}\}$; it has centralizer $U(2n-1)$
in $G$.  For $G/H = SO(2n)/SO(2n-1)$ we use
$\xi_2 = \diag\{J, \dots , J\}$ where
$J = \left ( \begin{smallmatrix} 0 & 1 \\ -1 & 0 \end{smallmatrix} \right )$;
it has centralizer $U(n)$ in $G$.  For $G/H = Spin(7)/G_2$ we consider
$\mathfrak{spin}(5) \oplus \mathfrak{spin}(2) \subset \mathfrak{g}$
and take $0 \ne \xi_3 \in \mathfrak{spin}(2)$. 
\smallskip

Now consider the noncompact cases listed in Proposition \ref{irred-flags-nonc}.
Going case by case, $\mathfrak{g}$ contains an appropriate multiple of 
the $\xi_i \in \mathfrak{g}_u$ of the previous paragraph, with the single 
exception of the spaces $SO(2p+1,2q+1)/SO(2p+1,2q)$.  That completes the
proof of Theorem \ref{abs-simple-summary}.
\end{proof}

\begin{remark}\label{q} {\rm 
In the case
$(G,H) = (G,\{1\})$, $M$ is the group manifold, the metric is any nonzero
multiple of the Killing form, $G$ acts on itself by left translation,
and $\xi$ can be any element of $\mathfrak{g}$ because it is centralized
by all right translations.  In this case $\xi^M$ is of constant length
without the requirement that $\xi$ be elliptic.}
\hfill $\diamondsuit$
\end{remark}

\section{Classification for $G$ complex simple}\label{sec5}
\setcounter{equation}{0}

We now look at the case where $G$ is simple but $G_\C$ is not. That is when
$G$ is the underlying real structure of a complex simple Lie group $E$; then
$G_\C = E \times \overline{E}$ where $\overline{E}$ is the complex conjugate
of $E$ and $G \hookrightarrow G_\C$ is the diagonal
$\delta E \hookrightarrow G_\C$\,.  It is
convenient to use the following very general lemma, which is based
on the infinitesimal version of \cite[Th\' eor\` eme 1]{W1960}.

\begin{lemma} \label{centralizer-trans}
Let $(M,ds^2)$ be any connected pseudo--riemannian homogeneous space.  Let
$\xi \in \mathfrak{g}$.  If the centralizer
$L:=\{g \in I(M,ds^2) \mid \Ad(g)\xi = \xi\}$ of $\xi$ in the isometry group
$I(M,ds^2)$ has an open orbit on $M$ then $\xi^M$ has constant length on $M$.
In particular if $L$ is transitive
on $M$ then $\xi^M$ has constant length on $M$.
\end{lemma}

\begin{proof} Let $\mathcal{O}$ be an open $L$--orbit on $M$.
If $x, y \in \mathcal{O}$, say $gx = y$ with $g \in L$,
then $ds^2(\xi^M_y , \xi^M_y) = ds^2(dg(\xi^M_x) , dg(\xi^M_x))
= ds^2(\xi^M_x , \xi^M_x)$.  Thus $||\xi^M||^2$ is constant on
$\mathcal{O}$.  As $||\xi^M||^2$ is real analytic on $M$ it is
constant.
\end{proof}

\begin{theorem}\label{not-abs-simple}
Suppose that $G$ is simple but $G_\C$ is not.  Then there is a 
nonzero elliptic element $\xi \in \mathfrak{g}$ such that the Killing 
vector field $\xi^M$ on the normal homogeneous space
$M = G/H$ has constant length, if and only if, up to finite covering,
$(G,H)$ is one of the pairs $(1) \,\,(SL(2n;\mathbb{C}),Sp(n;\mathbb{C}))$,
$(2) \,\,(SO(2n;\mathbb{C}),SO(2n-1;\mathbb{C}))$, 
$(3) \,\,(Spin(7;\mathbb{C}),(G_2)_{\mathbb{C}})$, or $(4)$ the group
manifold pair $(G,\{1\})$.
\end{theorem} 

\begin{remark}\label{r} {\rm In all cases of Theorem \ref{not-abs-simple},
$G/H$ is a complex affine algebraic variety.  Also, in the case
$(G,H) = (G,\{1\})$, $M$ is the group manifold, the metric is any nonzero
multiple of the Killing form, $G$ acts on itself by left translation, 
and $\xi$ can be any element of $\mathfrak{g}$ because it is centralized
by all right translations.  In this case $\xi^M$ is of constant length
without the requirement that $\xi$ be elliptic.}
\hfill $\diamondsuit$
\end{remark}

\begin{proof} 
Let $\xi \in \mathfrak{g}$ be nonzero and elliptic.  We may assume that it 
is contained in the Lie algebra $\mathfrak{k}$ of a maximal compact subgroup 
$K$ of $G$.  The point is that it is contained in a fundamental
(maximally compact) Cartan subalgebra of $\mathfrak{g}$.  All such Cartan
subalgebras are $\Ad(G)$--conjugate, and the $\mathfrak{g}$--centralizer of
any Cartan subalgebra of $\mathfrak{k}$ is one of them.  Thus, for the 
proof, we may assume $\xi \in \mathfrak{k}$.
\smallskip

Note that $K$ is a compact real form
when $G$ is regarded as a complex simple group.  Passing to a conjugate, 
$H$ is stable under the complex conjugation $\tau$ of $G$ with fixed point
set $K$, for $\tau$ is a Cartan involution.  Now 
$\mathfrak{g} = \mathfrak{k} + i\mathfrak{k}$ under $\tau$ and
$\mathfrak{h} = \mathfrak{h}\cap\mathfrak{k} + \mathfrak{h}\cap i\mathfrak{k}$.
These are orthogonal decompositions relative to the Killing form of $G$,
and the invariant bilinear form $b$ is a positive multiple of that Killing form.
\smallskip

Suppose that $\xi^M$ has constant length, equivalently that
$b(\mathrm{pr}_\mathfrak{h}(\mathrm{Ad}(g)\xi)),
\mathrm{pr}_\mathfrak{h}(\mathrm{Ad}(g)\xi)))$ is constant for $g \in G$.
Then $b(\mathrm{pr}_{\mathfrak{h}\cap\mathfrak{k}}(\mathrm{Ad}(g)\xi),
\mathrm{pr}_{\mathfrak{h}\cap\mathfrak{k}}(\mathrm{Ad}(g)\xi))$ is
constant for $g \in K$.  In other words $\xi$ defines a constant length
Killing vector field on $K/(K\cap H)$.
\smallskip

If $\mathfrak{k}\cap\mathfrak{h} = \mathfrak{k}$ then $\mathfrak{k}
\subset \mathfrak{h}$.  The adjoint action of $\mathfrak{k}$ on
$\mathfrak{g}$ is the sum of two copies of the adjoint representation
of $\mathfrak{k}$, which is irreducible, so $\mathfrak{k}
\subset \mathfrak{h}$ says that either $\mathfrak{h} = \mathfrak{k}$
or $\mathfrak{h} = \mathfrak{g}$.  If $\mathfrak{h} = \mathfrak{k}$
there is no nonzero Killing vector field of constant length on $G/H$.
If $\mathfrak{h} = \mathfrak{g}$ then $G/H$ is reduced to a point.
So $\mathfrak{k}\cap\mathfrak{h} \ne \mathfrak{k}$.
\smallskip

If $\mathfrak{k}\cap\mathfrak{h} = 0$ then $b(\mathfrak{k},\mathfrak{h}) = 0$
so $\xi \in \mathfrak{k} \subset \mathfrak{m}$.  Then 
$\mathrm{pr}_\mathfrak{m}(\xi) = \xi$ and
$\mathrm{pr}_\mathfrak{h}(\xi) = 0$.  In particular
$b(\mathrm{pr}_\mathfrak{m}(\Ad(g)\xi), \mathrm{pr}_\mathfrak{m}(\Ad(g)\xi))
= b(\xi,\xi)$  and $b(\mathrm{pr}_\mathfrak{h}(\Ad(g)\xi), 
\mathrm{pr}_\mathfrak{h}(\Ad(g)\xi)) = 0$ for all $g \in G$.
Now $\Ad(G)\xi \subset \mathfrak{m}$, so $\mathfrak{m}$ contains a nonzero
ideal of the simple Lie algebra $\mathfrak{g}$.  In other words $\mathfrak{m}
= \mathfrak{g}$ and $\mathfrak{h} = 0$, so $M$ is the group manifold $G$.
\smallskip

Now suppose  $\mathfrak{k}\cap\mathfrak{h} \ne 0$.  As 
$\mathfrak{k}\cap\mathfrak{h} \subsetneqq \mathfrak{k}$ and $\xi$ defines 
a constant length Killing vector field on $K/(K\cap H)$, 
we know from \cite{XW2014} or from Theorem \ref{abs-simple-summary}
that $(\mathfrak{k}, \mathfrak{k}\cap\mathfrak{h})$ is one of
\begin{equation}\label{3-chances}
{\rm (a)}\, (\mathfrak{su}(2n),\mathfrak{sp}(n)), \,\,
{\rm (b)}\, (\mathfrak{so}(2n),\mathfrak{so}(2n-1)), \, \text{ or }
{\rm (c)}\, (\mathfrak{so}(7),\mathfrak{g}_2).
\end{equation}
Here $(\mathfrak{h}, \mathfrak{k}\cap\mathfrak{h})$ is a symmetric pair,
$\mathfrak{h}\cap\mathfrak{k}$ is simple by (\ref{3-chances}), and
of course $\mathfrak{k} \ne \mathfrak{h} \subset \mathfrak{g}$.  
\smallskip

Decompose $\mathfrak{h} = \mathfrak{h}_1 \oplus \mathfrak{h}_2$ where
$\mathfrak{h}_2 \cap \mathfrak{k} = 0$ and every ideal of $\mathfrak{h}_1$
has nonzero intersection with $\mathfrak{k}$.  Then $\mathfrak{h}_2 \subset
\mathfrak{p}$ and $\mathfrak{h}\cap\mathfrak{k} = 
\mathfrak{h}_1\cap\mathfrak{k}$.  Thus (\ref{3-chances})
limits the possibilities of $\mathfrak{h}_1$ to
\begin{equation}\label{abc}
\begin{aligned}
{\rm (\ref{3-chances}a):}\,\,\, & \mathfrak{h}_1 = 
	{\rm (i)}\,\,\,\mathfrak{sp}(n;\mathbb{C}),\,
	{\rm (ii)}\,\,\, \mathfrak{sl}(n;\mathbb{H}), \text{ or }
	{\rm (iii)}\,\,\, \mathfrak{e}_{6,\mathfrak{c}_4} \text{ with } n=4 \\
{\rm (\ref{3-chances}b):}\,\,\, & \mathfrak{h}_1 =
	{\rm (iv)}\,\,\,\mathfrak{so}(2n-1;\mathbb{C}), \,
	{\rm (v)}\,\,\, \mathfrak{sl}(2n-1;\mathbb{R}), \text{ or }
        {\rm (vi)}\,\,\, \mathfrak{f}_{4;\mathfrak{b}_4} \text{ with } n = 5 \\
{\rm (\ref{3-chances}c):}\,\,\, & \mathfrak{h}_1 =
        {\rm (vii)}\,\,\, \mathfrak{g}_{2,\mathbb{C}}
\end{aligned}
\end{equation}
We eliminate case (iii) of (\ref{abc}) 
because $\mathfrak{e}_6$ has no nontrivial
representation of degree $8$, and (vi) because $\mathfrak{f}_4$ has no
nontrivial representation of degree $10$.  In case (v), passing to
the complexification we would have $\mathfrak{sl}(2n-1;\mathbb{C}) 
\subset \mathfrak{so}(2n;\mathbb{C}) \oplus \mathfrak{so}(2n;\mathbb{C})$
while $\mathfrak{sl}(2n-1;\mathbb{C})$ has no nontrivial orthogonal 
representation of degree $2n$; that eliminates case (v).  
At this point we notice that $\mathfrak{h}_1$ is a maximal subalgebra of 
$\mathfrak{g}$, so $\mathfrak{h} = \mathfrak{h}_1$\,.
\smallskip

Case (ii) is more delicate.  The analog of \cite{XW2014} reduces the existence 
of a Killing vector field $\xi$ of constant length to the question of whether
$\xi' = i\left ( \begin{smallmatrix} -1 & 0 \\ 0 & 1 \end{smallmatrix} \right )$
defines a Killing vector field of constant length on 
$M' = SL(2;\mathbb{C})/SL(1;\mathbb{H})$.  Since $M'$ is the noncompact
Riemannian 
symmetric space $SL(2;\mathbb{C})/SU(2)$, the answer is negative.  We have
eliminated cases (ii), (iii), (v) and (vi) of (\ref{abc}), and we have
shown $\mathfrak{h} = \mathfrak{h}_1$\,.
\smallskip

At this point we have shown that there is a nonzero elliptic
$\xi \in \mathfrak{g}$ such that $\xi^M$ has constant length, if and only if
$(G,H)$ is one of the four pairs listed in Theorem \ref{not-abs-simple}.
If $(G,H) = (SL(2n;\mathbb{C}),Sp(n;\mathbb{C}))$ we can take
$\xi = i\diag\{2n-1; 1,\dots , 1\}$; it is centralized by $GL(2n;\mathbb{C})$.
If $(G,H) = (SO(2n;\mathbb{C}),SO(2n-1;\mathbb{C}))$ we can take
$\xi = \diag\{J, \dots , J\}$ where $J =               
\left ( \begin{smallmatrix} 0 & 1 \\ -1 & 0 \end{smallmatrix} \right )$; it is
centralized by $GL(n;\mathbb{C})$.  If $(G,H) = 
(Spin(7;\mathbb{C}),(G_2)_{\mathbb{C}})$ we can take $\xi$ in a Cartan
subalgebra dual to a short root.  As noted in Remark \ref{r}, if
$(G,H) = (G,\{1\})$ we can take $\xi$ to be any element of the Lie
algebra $\mathfrak{g}$ acting by right translations.
Looking at the compact versions, in all cases one calculates 
$\dim Z_G(\xi)/(Z_G(\xi)\cap H) = \dim G/H$, so Lemma \ref{centralizer-trans}
ensures that the Killing vector field $\xi^M$ has constant length. 
\end{proof}

\begin{remark}\label{s}{\rm Here is another argument to eliminate case (ii)
of (\ref{abc}) in the proof of Theorem \ref{not-abs-simple}.
$\mathfrak{g}_\C \cong \mathfrak{g} \oplus \overline{\mathfrak{g}}$ with
$\mathfrak{g}$ embedded diagonally, so $\mathfrak{g}_\C$ has compact real
form $\mathfrak{g}_u \cong \mathfrak{k} \oplus \mathfrak{k}$ with
$\mathfrak{k}$ embedded diagonally.  Now $\xi \in \mathfrak{k}$ has form
$\xi = (\xi',\xi')$ inside $\mathfrak{g}_\C$, so it has nontrivial
projections to each of the two simple summands of $\mathfrak{g}_u$\,.
This is impossible here because $H_u$ is the diagonal $SU(2n)$ inside
$G_u \cong SU(2n)\times SU(2n)$.
}\hfill $\diamondsuit$
\end{remark}

\section{Summary for $G$ Simple}\label{sec6}
\setcounter{equation}{0}

Combining Theorems \ref{abs-simple-summary} and \ref{not-abs-simple} we
arrive at

\begin{theorem}\label{simple-summary}
Suppose that $G$ is simple.  Then there is a
nonzero elliptic element $\xi \in \mathfrak{g}$ such that the Killing
vector field $\xi^M$ on the normal homogeneous space
$M = G/H$ has constant length, if and only if, up to finite covering,
$(G,H)$ is one of the following.
\smallskip

\noindent {\rm \phantom{X} 1.}
$(SU(2p,2q),Sp(p,q))$ with $p + q = n$,
$(SL(2n;\R),Sp(n;\R))$ or 
$(SL(2n;\C),Sp(n;\C)$
\vskip 2pt

\noindent {\rm \phantom{X} 2.}
$(SO(2p+2,2q),SO(2p+1,2q))$ with $p+q = n$ or
$(SO(2n+2;\C),Sp(2n+1;\C)$
\smallskip

\noindent {\rm \phantom{X} 3.}
$(Spin(7),G_2)$,
$(Spin(3,4),(G_2)_\R)$ or
$(Spin(7;\C),G_{2,\C})$
\smallskip

\noindent {\rm \phantom{X} 4.} $(G,\{1\})$
\end{theorem}

Looking through this listing one sees
\begin{corollary}\label{homog}
Suppose that $G$ is simple, and that $\xi \in \mathfrak{g}$ is nonzero and
elliptic.  Let $L$ be the centralizer of $\xi$ in $G$.  Then the
following are equivalent.

\noindent {\rm \phantom{X} 1.} $\xi^M$ has constant length on $M = G/H$. 

\noindent {\rm \phantom{X} 2.} $L$ has an open orbit on $G/H$.

\noindent {\rm \phantom{X} 3.} $H$ has an open orbit on the flag domain $G/L$.
\end{corollary}

\section{The Three Cases}\label{sec7}
\setcounter{equation}{0}
Retain the notation of Section \ref{sec2}.  Note that $G_u$ acts
transitively on the complex flag manifold $Z = G_\C/Q$, so $Z = G_u/L_u$ 
where $L_u$ is a compact real form of $L$.  This
expresses $Z$ as a compact simply connected homogeneous K\" ahler manifold.
\smallskip

By {\em coset space reduction} of $G/H$ we mean a decomposition
$G = G' \times G''$ (locally) such that $H = (H \cap G') \times (H \cap G'')$,
and consequently $G/H = (G'/(H \cap G')) \times (G''/(H \cap G''))$,
with each factor of positive dimension..  We will say that $G/H$ is 
{\em coset space irreducible} if there is no such nontrivial reduction.
The following is immediate from the definitions.

\begin{lemma}\label{normal-reduction1}
Suppose that $G$ is semisimple.
Let $G/H = G'/H' \times G''/H''$ be a coset space reduction.  
If $b$ is an invariant bilinear form on $\mathfrak{g}$ then $b = b' \oplus b''$ where $b'$ (resp. $b''$) is an invariant bilinear form on $\mathfrak{g}'$
(resp. $\mathfrak{g}''$).  The corresponding decomposition 
$\mathfrak{g} = \mathfrak{h} + \mathfrak{m}$ breaks up as
$\mathfrak{g}' = \mathfrak{h}' + \mathfrak{m}'$ and
$\mathfrak{g}'' = \mathfrak{h}'' + \mathfrak{m}''$ where
$\mathfrak{h} = \mathfrak{h}' \oplus \mathfrak{h}''$,
$\mathfrak{m} = \mathfrak{m}' \oplus \mathfrak{m}''$,
$\mathfrak{m}'$ is the $b'$--orthocomplement of $\mathfrak{h}'$ in
$\mathfrak{g}'$, and $\mathfrak{m}''$ is the $b''$--orthocomplement of 
$\mathfrak{h}''$ in $\mathfrak{g}''$.  In particular the corresponding
factors of the pseudo--riemannian product decomposition are normal
homogeneous spaces.
\end{lemma}

Let $M = G/H$ with $G$ reductive.
Up to finite coverings we then have a decomposition
\begin{equation}\label{g-split1}
G = G_0 \times G_1 \times\dots\times G_s \times G_{s+1} \times\dots\times G_r
\end{equation}
where $G_0$ is commutative, $G_i$ is simple for $i > 0$.  $H$ is the
(isomorphic) image of a reductive Lie group $\widetilde{H}$ under a
homomorphism $\varphi(x) = (\varphi_0(x), \dots , \varphi_r(x))$ where
$\varphi: \widetilde{H} \to G_i$\,.
We are going to study constant length Killing vector fields on $M = G/H$ 
defined by vectors 
\begin{equation}\label{split-xi}
\text{$\xi = \xi_0 + \dots + \xi_r \in \mathfrak{g}$,
$\xi_i \in \mathfrak{g}_i$\,, $\xi_i \ne 0$ for $1 \leqq i \leqq s$,
and $\xi_i = 0$ for $s < i \leqq r$.}
\end{equation}
In view of Lemma \ref{normal-reduction1} we need only consider the
coset irreducible cases.  There are three basic possibilities of reductive
normal coset irreducible $G/H$:
\begin{equation}\label{cases}
\begin{aligned}
{\rm (i)}& \text{ for some index $i$ we have $\varphi_i(\widetilde{H}) 
	= \{1\}$,} \\
{\rm (ii)}& \text{ for every index $i$ we have $\{1\} \neq 
	\varphi_i(\widetilde{H}) \subsetneqq G_i$\,, and}\\
{\rm (iii)}& \text{ for some index $i$ we have 
	$\varphi_i(\widetilde{H}) = G_i$\,.}
\end{aligned}
\end{equation}
The first of these cases is somewhat trivial:

\begin{lemma} Let $M = G/H$ be coset space irreducible with some
$\varphi_i(\widetilde{H}) = \{1\}$ then $G = G_i = M$
and every $\xi \in \mathfrak{g}_i$ defines a constant length Killing 
vector field on $M$.
\end{lemma}

\begin{proof} The hypothesis says that $G_i = G_i/\varphi_i(\widetilde{H})$
is a factor in a coset space reduction of $G/H$, and coset space
irreducibility says that $G_i = G_i/\varphi_i(\widetilde{H})$ must be all
of $G/H$.  As given, $G_i$ acts isometrically on itself by left translations,
and by normality the right translations also are isometries. 
If $\xi \in \mathfrak{g}$ comes from the left action of $G$
on itself, it is centralized by the right action, which is transitive, 
so the corresponding vector field $\xi^M$ has constant length.
\end{proof}

Now we may (and do) assume that each $\dim \varphi_i(\widetilde{H}) > 0$.
The second case is 

\begin{theorem} \label{reduce-to-simple}
Assume that $M = G/H$ is a coset space irreducible normal homogeneous
space with $G$ semisimple and $H$ reductive in $G$.  In the notation of
{\rm (\ref{g-split1})} suppose that $\varphi_i(\widetilde{H})
\subsetneqq G_i$ and $\dim \varphi_i(\widetilde{H}) > 0$ for each $i > 0$.
Let $\xi = \xi_0 + \dots + \xi_r \in \mathfrak{g}$, elliptic and decomposed
as in {\rm (\ref{split-xi})}.  
Consider the following conditions.

{\rm (1)} $\xi$ defines a constant length Killing vector field $\xi^M$ 
on $M = G/H$,

{\rm (2)} For each $i$, $\xi_i$ defines a constant length Killing vector field 
$\xi_i^{M_i}$ on $M_i = G_i/\varphi_i(\widetilde{H})$.

{\rm (3)} For each $i$, $\xi_i$ defines a constant length Killing vector field 
$\xi_i^M$ on $M$.

{\rm (4)} The $\Ad(G)$--centralizer of $\xi$ has an open orbit on $M$.

\noindent
Then 

(a) {\rm (1)} implies {\rm (2)} but {\rm (2)} does not imply {\rm (1)}; 

(b) {\rm (2)} and {\rm (3)} are equivalent; and 

(c) {\rm (1)} and {\rm (4)} are equivalent.  
\end{theorem}

\begin{proof} 
As in the first paragraph of the proof of Theorem \ref{not-abs-simple} 
we may assume that $\xi$ is contained in
$\mathfrak{k} = \mathfrak{g} \cap \mathfrak{g}_u$\,. 
\smallskip

We write $L$ and $L_u$ for the
respective centralizers of $\xi$ in $G$ and $G_u$ and $Z$ for the complex
flag manifold $Z = G_u/L_u = G_\C/Q$.  We also write $L_i$ and $L_{i,u}$ for
the respective centralizers of $\xi_i$ in $G_i$ and $G_{i,u}$, so
$Z$ is the product of the $Z_i = G_{i,u}/L_i$\,, where of course
$L_0 = G_0$ and $L_i = G_i$ for $i > s$, so those $Z_i$ are single points.
\smallskip

We first prove that (1) implies (2) and (4).  As a subgroup of $G$,  $H_u$ acts 
holomorphically and isometrically on $Z$\,.  Here $Z$ carries the 
$G_u$--invariant Kaehler metric defined by its complex structure as
$G_C/Q$ and its normal Riemannian metric from the negative of the
Killing form of $G_u$\,.  The action is Hamiltonian. 
We are assuming (1), in other words that $\xi^M$ has constant length
on $M = G/H$, so Proposition \ref{carryover} says that $\xi^{M_u}$ has
constant length on $M_u = G_u/H_u$\,.  In other words the momentum
map for the action of $H_u$ on $Z$ has constant square norm.
Thus \cite[Theorem 1]{GP2004} $Z = Z' \times Z''$, holomorphically and 
isometrically, where $Z'$ and $Z''$ are complex flag manifolds  such that
$H_u$ is transitive on $Z'$ and $H_u$ acts trivially on $Z''$.
\smallskip

The group $H_u$ acts nontrivially on $Z_i$ for $1 \leqq i \leqq s$.
For if the action were trivial then $\varphi_i(\widetilde{H_u})$ would be 
normal in $G_{i,u}$\,, while
it is $\ne \{1\}$, forcing $\varphi_i(\widetilde{H_u}) = G_{i,u}$\,. 
This possibility was excluded by hypothesis.
Thus $Z' = Z_1 \times\dots\times Z_s$\,.
Now set $G' = G_1 \times\dots\times G_s$\,, $L' = L_1 \times\dots\times L_s$\,,
$\varphi' = \varphi_1 \times\dots\times \varphi_s$ and 
$H' = \varphi'(\widetilde{H})$.  Then $H_u'$ is transitive on $Z'$.
It follows from \cite[Proposition 2.1]{W2001} that 
$H'_{i,u} := \varphi_i(\widetilde{H}_u)$ is transitive on $Z_i$ for 
$1 \leqq i \leqq s$.  Equivalently $G_u' = H'_uL'_u$\,, which is the same
(take inverses) as $G_u' = L'_u H'_u$\,, so $L'_u$ is transitive on $Z'$.
In particular $G_{u,i} = L'_{u,i} H'_{u,i}$\,.  Thus
$\xi_i^{M_{i,u}}$ has constant length on $M_{i,u}$ for 
$1 \leqq i \leqq s$.  Thus (1) implies (2) and (4), and 
(4) implies (1) by Lemma \ref{centralizer-trans}.
\smallskip

It is obvious that (3) implies (2).  Given (2), the centralizer $L_i$ of
$\xi_i$ in $G_i$ is transitive on $M_i$, so the centralizer of $\xi_i$
in $G$ is transitive on $M$, and (3) follows.
\smallskip

It remains only to show that (2) does not imply (1).
Consider the case $G = SO(2n)\times SO(2n)$ with $H = SO(2n-1)$ embedded
diagonally and $\xi = \diag\{J, \dots, J\}$ where 
$J = \left ( \begin{smallmatrix} 0 & 1 \\ -1 & 0 \end{smallmatrix}\right )$.
Then $L = U(n)\times U(n)$ and $L \cap H = U(n-1)$ so
$\dim G/H = 2n^2+n-1 > 2n^2 = \dim L$, so $L$ cannot have an open orbit
on $G/H$ when $n > 1$.  On the other hand the projections 
of $\xi$ to the ideals of $\mathfrak{g}$ define constant length Killing
vector fields $\xi_i^{M_i}$ on the $M_i = G_i/\varphi_i(\widetilde{H})$
because $L_i = U(n)$ is transitive on $M_i = G_i/\varphi_i(\widetilde{H})
= SO(2n)/SO(2n-1) = S^{2n-1}$.  Thus (2) does not imply (4).  But (1) and (4)
are equivalent, so (2) does not imply (1).
\end{proof}

The third case includes the pseudo--riemannian 
group manifolds $(H \times H)/(diag\{H\})$
for real simple Lie groups $H$, but the following example shows that this
case is more of a combinatorial problem than a geometric or
Lie theoretic problem.

\begin{example}\label{double}{\rm
Let $G'$ and $G''$ be reductive Lie groups.
Let $\widetilde{H}$ be reductive with homomorphisms
$\varphi': \widetilde{H} \to G'$ and $\varphi'': \widetilde{H} \to G''$
such that $h \mapsto (\varphi'(h), \varphi''(h))$ is an isomorphism 
of $\widetilde{H}$ onto a reductive
subgroup $H$ of $G := G' \times G''$.  Let $M = G/H$ be the corresponding 
homogeneous space with any $G$--invariant pseudo--riemannian metric.
Suppose that $\xi \in \mathfrak{g}'$ and that $\varphi'(\widetilde{H}) = G'$.  
Then $G''$ centralizes $\xi$ and $G = HG''$, so the centralizer of $\xi$
in $G$ is transitive on $M$.  Thus $\xi^M$ has constant length on $M$.
The most familiar case of this is a compact group manifold
$(H \times H)/(diag\{H\})$.} \hfill $\diamondsuit$
\end{example}
\vfill\pagebreak

\end{document}